\documentclass{amsart}
\usepackage{amsfonts, amssymb}
\usepackage{enumerate}
\usepackage{graphicx}
\usepackage{amsmath}
\newtheorem{theorem}{Theorem}

\newtheorem{corollary}[theorem]{Corollary}

\newtheorem{lemma}[theorem]{Lemma}

\newtheorem{remark}[theorem]{Remark}

\DeclareMathOperator\End{End}

\begin{document}
\title{Decomposing elements of a right self-injective ring}
\author{Feroz Siddique}
\address{Department of Mathematics and Computer Science, St. Louis University, St.
Louis, MO-63103, USA}
\email{fsiddiq2@slu.edu}
\author{Ashish K. Srivastava}
\address{Department of Mathematics and Computer Science, St. Louis University, St.
Louis, MO-63103, USA}
\email{asrivas3@slu.edu}


\begin{abstract}
It was proved independently by both Wolfson [An ideal theoretic characterization of the ring of all
linear transformations, Amer. J. Math. 75 (1953), 358-386] and Zelinsky [Every Linear Transformation is Sum of Nonsingular Ones,
Proc. Amer. Math. Soc. 5 (1954), 627-630] that every linear transformation of a vector space $V$ over a division ring $D$ is the sum of two invertible linear transformations except when $V$ is one-dimensional over $\mathbb Z_2$. This was extended by Khurana and Srivastava [Right self-injective rings in which each element is sum of two units, J. Algebra and its Appl., Vol. 6, No. 2 (2007), 281-286] who proved that every element of a right self-injective ring $R$ is the sum of two units if and only if $R$ has no factor ring isomorphic to $\mathbb Z_2$. In this paper we prove that if $R$ is a right self-injective ring, then for each element $a\in R$ there exists a unit $u\in R$ such that both $a+u$ and $a-u$ are units if and only if $R$ has no factor ring isomorphic to $\mathbb Z_2$ or $\mathbb Z_3$. 
\end{abstract}

\maketitle

\noindent All our rings are associative with identity element $1$. Following V\'{a}mos \cite{Vamos}, an element $x$ in a ring $R$ is called a {\it $k$-good element} if $x$ can be expressed as the sum of $k$ units in $R$. A ring $R$ is called a {\it $k$-good ring} if each element of $R$ is a $k$-good element. Many authors including Chen \cite{Chen1}, Ehrlich \cite{Ehrlich}, Henriksen \cite{Henriksen1}, Fisher - Snider \cite{FS}, Khurana - Srivastava (\cite{KS1}, \cite{KS2}), Raphael \cite{Raphael}, V\'{a}mos (\cite{Vamos}, \cite{VW}), Wiegand \cite{VW} and Wang - Zhou \cite{WZ} have studied rings generated additively by their unit elements, in particular, $2$-good rings. We refer the readers to \cite{Sri} for a survey of rings generated by units.

\bigskip 

\noindent In \cite{Perkins} a ring $R$ is said to be a {\it twin-good ring} if for each $x\in R$ there exists a unit $u\in R$ such that both $x+u$ and $x-u$ are units in $R$. Clearly every twin-good ring is $2$-good. However, there are numerous examples of $2$-good rings which are not twin-good. For example, $\mathbb Z_3$ is $2$-good but not twin-good. We denote by $J(R)$, the Jacobson radical of ring $R$ and by $U(R)$, the group of units of $R$. 

The following observations were noted in \cite{Perkins}. Their proofs are straightforward.

\begin{lemma} \label{divring}
If $D$ is a division ring such that $|D|\ge 4$, then $D$ is twin-good.
\end{lemma}

\begin{lemma} \label{radtwin}
For a ring $R$, we have the following:
\begin{enumerate}[(i)]
\item If $R$ is twin-good then for any proper ideal $I$ of $R$, the factor ring $R/I$ is also twin-good.
\item If a factor ring $R/I$ is twin-good and $I\subseteq J(R)$, then $R$ is twin-good. Thus, in particular, it follows that a ring $R$ is twin-good if and only if $R/J(R)$ is twin-good. 
\item If $R$ is a direct product of rings $R_i$ where each $R_i$ is a twin-good ring, then $R$ is also a twin-good ring.
\end{enumerate}
\end{lemma}

\section{Main Results}

\noindent A ring $R$ is called \textit{right self-injective} if each right $R$-homomorphism from any right ideal of $R$ to $R$ can be extended to an endomorphism of $R$. As the ring of linear transformations is a right self-injective ring, the result of Wolfson and Zelinsky attracted quite a bit of attention toward understanding which right self-injective rings are $2$-good.

\begin{theorem}$($V\'{a}mos \cite{Vamos}$)$
A right self-injective ring $R$ is $2$-good if $R$ has no non-zero corner ring that is Boolean. 
\end{theorem}

Khurana and Srivastava \cite{KS1} extended the result of Wolfson and Zelinsky to the class of right self-injective rings and proved the following

\begin{theorem} $($Khurana, Srivastava \cite{KS1}$)$
A right self-injective ring $R$ is $2$-good if and only if $R$ has no factor ring isomorphic to $\mathbb{Z}_{2}$. 
\end{theorem}

We will prove an analogue of this result for twin-good rings. But, first we have some definitions and useful lemmas.

We say that an $n\times n$ matrix $A$ over a ring $R$ \textit{admits a diagonal reduction} if there exist invertible matrices $P,Q\in \mathbb M_{n}(R)$ such that $PAQ $ is a diagonal matrix. Following Ara et. al. \cite{AGMP}, a ring $R$ is called an \textit{elementary divisor ring} if every square matrix over $R$ admits a diagonal reduction. This definition is less stringent than the one proposed by Kaplansky in \cite{Kap}. The class of elementary divisor rings includes unit-regular rings and von Neumann regular right self-injective rings (see \cite{AGMP}, \cite{Henriksen2}).

 If $R$ is an elementary divisor ring, then clearly the matrix ring $\mathbb M_n(R)$ is $2$-good for each $n\ge 2$. In the case of twin-good rings, we have the following  
 
\begin{lemma} \label{edr}
Let $R$ be an elementary divisor ring. Then the matrix ring $\mathbb M_n(R)$ is twin-good for each $n \ge 3$.
\end{lemma}

\begin{proof}
Let $n \in N$ such that $n \ge 3$. Let $M$ be any arbitrary element of $\mathbb M_n(R)$. Then there exist invertible matrices $E, F\in \mathbb M_n(R)$ such that $EMF$ is a diagonal matrix. Set $A=EMF$. Then $A\in \mathbb M_n(R)$ is a diagonal matrix. Suppose \[A=  \left[\begin{array}{ccccccc}
d_{1}&0&0&0\cdots&0&0  \\0&d_{2}&0&0\cdots&0&0\\  0&0&d_{3}&0\cdots&0&0 \\ 0&0&0&d_{4}\cdots&0&0 \\   \vdots  &\vdots   &\vdots &\vdots  &\vdots &\vdots  \\ 0&0&0&0\cdots&d_{n-1}&0\\0&0&0&0\cdots&0&d_{n}\\ \end{array}\right].\]

\noindent We consider the first (n-1) columns of the first row of $A$ and call it $P$. Thus $P$ is a $1\times (n-1)$ matrix given by $P=\left[\begin{array}{ccccccc}
d_{1}&0&0&0\cdots&0 \\ \end{array}\right]$. 
Similarly we consider the last (n-1) rows of the last column of $A$ and call it $Q$. Thus $Q$ is a $(n-1) \times 1$ matrix given by $Q=\left[\begin{array}{c}
0 \\ 0\\0\\ \vdots\\d_{n} \end{array}\right]$. Now we consider the lower left $(n-1)\times (n-1)$ block in $A$ and call it $B$. Thus $B=\left[\begin{array}{cccccc}
0&d_{2}&0&0&0\cdots&0  \\0&0&d_{3}&0&0\cdots&0\\  0&0&0&d_{4}&0\cdots&0 \\ 0&0&0&0&d_{5}\cdots&0 \\   \vdots  &\vdots   &\vdots &\vdots   &\vdots  \\ 0&0&0&0&0\cdots&d_{n-1}\\ \end{array}\right].
$

\noindent Let $T$ = $QP + I_{(n-1)}$. Then $T=\left[\begin{array}{ccccc}
1&0&0&0\cdots&0 \\0&1&0&0\cdots&0\\  0&0&1&0\cdots&0 \\ 0&0&0&1\cdots&0 \\   \vdots  &\vdots   &\vdots &\vdots   &\vdots  \\ d_n\*d_1&0&0&0\cdots&1\\ \end{array}\right] \in \mathbb M_{n-1}(R)$.

\noindent  Now we create an $n\times n$ matrix $U = \left[
\begin{array}{cc}
O_1 & 1 \\
T & O_2
\end{array}
\right],
$
where $O_1$ and $O_2$ are $1\times (n-1)$ and  $(n-1)\times 1$ zero matrices, $1$ is the identity in $R$ and $T$ is the $(n-1)\times (n-1)$ matrix created above.
Then \[U=\left[\begin{array}{cccccc}
0&0&0&0\cdots&0&1 \\1&0&0&0\cdots&0&0\\  0&1&0&0\cdots&0&0 \\ 0&0&1&0\cdots&0&0 \\   \vdots  &\vdots   &\vdots &\vdots   &\vdots &\vdots \\ d_n\*d_1&0&0&0\cdots&1&0\\ \end{array}\right]\in \mathbb M_n(R).\]

\noindent Clearly $U$ is a unit in $\mathbb M_{n}(R)$ whose inverse is given by $U^{-1}=[a_{i,j}]$, where $a_{i,i+1}=1$, $a_{n-1, 2}=-d_n\*d_1$, $a_{n,1}=1$, and $a_{i,j}=0$ elsewhere.

\noindent  Now we consider the matrices $A+U, A-U$ in $\mathbb M_{n}(R)$ which are of the form \[\left[\begin{array}{cccccc}
d_1&0&0&0\cdots&0&1 \\1&d_2&0&0\cdots&0&0\\  0&1&d_3&0\cdots&0&0 \\ 0&0&1&d_4\cdots&0&0 \\   \vdots  &\vdots   &\vdots &\vdots   &\vdots  \\ d_n\*d_1&0&0&0\cdots&1&d_n\\ \end{array}\right]
, \left[\begin{array}{cccccc}
d_1&0&0&0\cdots&0&-1 \\-1&d_2&0&0\cdots&0&0\\  0&-1&d_3&0\cdots&0&0 \\ 0&0&-1&d_4\cdots&0&0 \\   \vdots  &\vdots   &\vdots &\vdots   &\vdots  \\ -d_n\*d_1&0&0&0\cdots&-1&d_n\\ \end{array}\right]\] respectively.
 
 It can easily be checked that $A+U$ and $A-U$ are invertible matrices. Thus we have shown that there exists an invertible matrix $U \in \mathbb M_n(R)$ such that both $A+U$ and $A-U$ are invertible matrices. Clearly $E^{-1}UF^{-1}$ is also invertible in $\mathbb M_n(R)$ such that both $E^{-1}AF^{-1} + E^{-1}UF^{-1}$ and $E^{-1}AF^{-1} - E^{-1}UF^{-1}$ are invertible. Thus it follows that $M$ is twin-good. Hence the matrix ring $\mathbb M_n(R)$ is twin-good for each $n \ge 3$.
\end{proof}

It follows from the result of Wolfson and Zelinsky that any proper matrix ring $\mathbb M_n(D)$ is $2$-good where $D$ is a division ring and $n\ge 2$. For twin-good rings, we have the following.

\begin{lemma} \label{abelian}
If $R$ is an abelian regular ring, then $\mathbb M_2(R)$ is twin-good.
\end{lemma}

\begin{proof}
Let $A$ be an arbitrary element of $\mathbb M_{2}(R)$. As $R$ is an elementary divisor ring, there exist invertible matrices $P, Q\in \mathbb M_{2}(R)$ such that \[PAQ=\left[
\begin{array}{cc}
a & 0 \\
\noalign{\vspace{-0.3em}}
0 & b
\end{array}
\right] \]
for some $a, b\in R$. Since $R$ is abelian regular, there exist $u,v \in U(R)$ and central idempotents $e_{1},e_{2}\in R$ such that $a = e_{1}u , b = e_{2}v$. \\
Then we can write  $PAQ= UE$ where  
\[U=\left[
\begin{array}{cc}
u & 0 \\
\noalign{\vspace{-0.3em}}
0 & v
\end{array}
\right] \]  and \[E=\left[
\begin{array}{cc}
e_{1} & 0 \\
\noalign{\vspace{-0.3em}}
0 & e_{2}
\end{array} 
\right] .\] 
Clearly $U$ is a unit in $\mathbb M_2(R)$ and $E$ is an idempotent in $\mathbb M_2(R)$. We consider $V\in \mathbb M_2(R)$ of the form \[V=\left[
\begin{array}{cc}
0 & -1 \\
\noalign{\vspace{-0.3em}}
-1 & -e_{2}
\end{array}
\right] .\] 
Clearly the matrix $V$ is a unit with inverse \[V^{-1}=\left[
\begin{array}{cc}
e_2 & -1 \\
\noalign{\vspace{-0.3em}}
-1 & 0
\end{array}
\right] .\] 
 
\noindent Now we have \[E-V=\left[
\begin{array}{cc}
e_{1} & 1 \\
\noalign{\vspace{-0.3em}}
1 & 2e_{2}
\end{array}
\right] .\] Clearly $E-V$ is a unit with its inverse given by 
\[(E-V)^{-1}=\left[
\begin{array}{cc}
4e_{1}e_{2}-2e_{2} & 1-2e_{1}e_{2} \\
\noalign{\vspace{-0.3em}}
1-2e_{1}e_{2} & 2e_{1}e_{2}-e_{1}
\end{array}
\right]. \] 

We have \[E+V=\left[
\begin{array}{cc}
e_{1} & -1 \\
\noalign{\vspace{-0.3em}}
-1 & 0
\end{array}
\right] .\] Clearly $E+V$ is a unit with its inverse given by 
\[(E+V)^{-1}=\left[
\begin{array}{cc}
0 & -1 \\
\noalign{\vspace{-0.3em}}
-1 & -e_1
\end{array}
\right]. \]  

\noindent Thus we have obtained a unit $V$ such that both $E-V$ and $E+V$ are units. Clearly, then $UV,UE-UV,UE+UV$ are units in $\mathbb M_2(R)$. Thus $PAQ-UV$ and $PAQ+UV$ are units. Therefore $PAQ$ is twin-good and consequently, multiplying by $P^{-1}$ in left and $Q^{-1}$ in right, we conclude that $A$ is twin good. This shows that $\mathbb M_2(R)$ is also twin-good.
\end{proof}

\begin{corollary} \label{divmatrix}
If $R$ is an abelian regular ring, then the matrix ring $\mathbb M_n(R)$ is twin-good for each $n\ge 2$. 

In particular, if $D$ is a division ring, then the matrix ring $\mathbb M_n(D)$ is twin-good for each $n\ge 2$.
\end{corollary}

\begin{proof}
It is straightforward from Lemma \ref{edr} and Lemma \ref{abelian}.
\end{proof}

\begin{remark}
As a consequence of the above corollary, it follows that a semilocal ring $R$ is twin-good if and only if $R$ has no factor ring isomorphic to $\mathbb Z_2$ or $\mathbb Z_3$.  
\end{remark}

\bigskip

\noindent Now we are ready to prove our main theorem.

\begin{theorem}\label{main} A right self-injective ring $R$ is twin good if and only if $R$ has no factor ring isomorphic to $\mathbb Z_2$ or $\mathbb Z_{3}$.
\end{theorem}

\begin{proof}
Let $R$ be a right self-injective ring such that R has no factor ring isomorphic to $\mathbb Z_2$ or $\mathbb Z_3$. We know that $R/J(R)$ is a von Neumann regular right self-injective ring. From the type theory of von Neumann regular right self-injective rings it follows that $R/J(R) \cong R_1 \times R_2 \times R_3 \times R_4 \times R_5$ where $R_1$ is of type $I_f$, $R_2$ is of type $I_\infty$, $R_3$ is of type $II_f$, $R_4$ is of type $II_\infty$, and $R_5$ is of type $III$ (see \cite[Theorem 10.22]{Goodearl}).  Taking $T=R_2 \times R_4 \times R_5$, we may write $R/J(R) \cong R_1 \times R_3 \times T$, where $T$ is purely infinite. We have $T_T\cong nT_T$ for all positive integers $n$ by \cite[Theorem 10.16]{Goodearl}. In particular, for $n=3$, this yields $T \cong \mathbb M_{3}(T)$. Since $T$ is an elementary divisor ring, by Lemma \ref{edr} we conclude that $\mathbb M_{3}(T)$ is twin-good and consequently $T$ is twin-good.

Next we consider $R_1$. We know that $R_1 \cong \prod\mathbb{M}_{n_i}(S_i)$ where each $S_i$ is an abelian regular self-injective ring (see \cite[Theorem 10.24]{Goodearl}).  Since each $S_i$ is an elementary divisor ring, we know $M_{n_i}(S_i)$ is twin good whenever $n_i \ge 3$. If $n_i=2$, then by Lemma \ref{abelian}, we have that $\mathbb M_{n_i}(S_i)$ is twin-good.

Consider $n_{i}$ = 1. Then we wish to prove that $S_i$ is twin-good. This was shown in \cite{Perkins} but we present the proof here for the sake of completeness. Assume to the contrary that $S_i$ is not twin-good. Then there exists an element $x\in S_i$ such that, for any $u\in U(S_i)$, either $x+u\not\in U(S_i)$ or $x-u\not\in U(S_i)$. Consider the set $\mathcal S=\{I: I$ is an ideal of $S_i$ such that $\bar{x}+\bar{u}\not\in U(S_i/I)$ or $\bar{x}-\bar{u}\not\in U(S_i/I)$, for each $u\in U(S_i)\}$. Clearly, $\mathcal S$ is a non-empty set. It may be shown that $\mathcal S$ is an inductive set and hence, by Zorn's lemma, $\mathcal S$ has a maximal element, say $M$. Clearly then $S_i/M$ is indecomposable as a ring and therefore it has no nontrivial central idempotent. Since $S_i/M$ is an abelian regular ring, this yields that $S_i/M$ has no nontrivial idempotent. Hence, $S_i/M$ is a division ring. Therefore, by Lemma \ref{divring}, it follows that $S_i/I\cong \mathbb Z_2$ or $S_i/I \cong \mathbb Z_3$. This yields a contradiction to our assumption. Hence, $S_i$ is twin-good.

We now consider $R_3$. Since $R_3$ is of type $II_f$, we can write $R_3 \cong n(e_{n}R_3)$ for each $n \in \mathbb N$ where $e_n$ is an idempotent in $R$ (see \cite[Proposition 10.28]{Goodearl}). In particular, for $n = 3$ we have $R_3 \cong M_{3}(e_{3}R_{3}e_{3}$). As $e_{3}R_{3}e_{3}$ is an elementary divisor ring, it follows that $M_{3}(e_{3}R_{3}e_{3})$ is twin good by Lemma \ref{edr}.

Thus $R/J(R)$, being a direct product of twin-good rings, is twin good. Hence, by Lemma \ref{radtwin}, $R$ is twin good.

The converse is obvious.
\end{proof}

As a consequence, we have the following

\begin{corollary}
For any linear transformation $T$ on a right vector space $V$ over a division ring $D$, there exists an invertible linear transformation $S$ on $V$ such that both $T-S$ and $T+S$ are invertible, except when $V$ is one-dimensional over $\mathbb Z_2$ or $\mathbb Z_3$. 
\end{corollary}

\begin{remark}
\begin{enumerate}[(i)]
\item \noindent Wang and Zhou in \cite{WZ} have shown that if $D$ is a division ring such that $|D|>3$, then for each linear transformation $T$ on a right vector space $V_D$, there exists an invertible linear transformation $S$ on $V_D$ such that $T+S$, and $T-S^{-1}$ are invertible.

\bigskip

\item Chen \cite{Chen2} has recently shown that if $V$ is a countably generated right vector space over a division ring $D$ where $|D|>3$, then for each linear transformation $T$ on $V_D$, there exist invertible linear transformations $P$ and $Q$ on $V_D$ such that $T-P$, $T-P^{-1}$ and $T^2-Q^2$ are invertible.    
\end{enumerate}
\end{remark}

\noindent Now we may adapt the techniques of \cite{KS1} and generalize our main result to the endomorphism rings of several classes of modules. Recall that a module $M$ is said to be {\it $N$-injective} if for every submodule $N_1$ of the module $N$, all homomorphisms $N_1\rightarrow M$ can be extended to homomorphisms $N\rightarrow M$. A right $R$-module $M$ is {\it injective} if $M$ is $N$-injective for every $N\in $ Mod-$R$. A module $M$ is said to be {\it quasi-injective} if $M$ is $M$-injective. 

\bigskip

\noindent Consider the following three conditions on a module $M$;\\
C1: Every submodule of $M$ is essential in a direct summand of $M$.\\
C2: Every submodule of $M$ isomorphic to a direct summand of $M$ is itself a direct summand of $M$.\\
C3: If $N_1$ and $N_2$ are direct summands of $M$ with $N_1\cap N_2=0$ then $N_1\oplus N_2$ is also a direct summand of $M$.

\bigskip

\noindent A module $M$ is called a {\it continuous module} if it satisfies conditions C1 and C2. A module $M$ is called {\it $\pi$-injective} (or {\it quasi-continuous}) if it satisfies conditions C1 and C3. 

\bigskip

\noindent A right $R$-module $M$ is said to satisfy the exchange property if for every right $R$-module $A$ and any two direct sum decompositions $A=M^{\prime}\oplus N=\oplus_{i\in I}A_{i}$ with $M^{\prime} \simeq M$, there exist submodules $B_i$ of $A_i$ such that $A=M^{\prime} \oplus (\oplus_{i \in I}B_i)$.\\
If this hold only for $|I|<\infty$, then $M$ is said to satisfy the finite exchange property. A ring $R$ is called an {\it exchange ring} if $R_R$ satisfies the (finite) exchange property.  

\noindent Now, we have the following for endomorphism ring of a quasi-continuous module
 
\begin{corollary}
Let $S$ be any ring, $M$ be a quasi-continuous right $S$-module with finite exchange property and $R=\End(M_S)$. If no factor ring of $R$ is isomorphic to $\mathbb{Z}_{2}$ or $\mathbb Z_3$, then $R$ is twin-good.
\end{corollary}

\begin{proof}
This proof is almost identical to the proof of \cite[Theorem 3]{KS1} but we write it here for the sake of completeness. Let $\Delta =\{f\in R:\ker $ $f\subset _{e}M\}$. Then $\Delta $ is an ideal of $R$ and $\Delta \subseteq J(R)$. By (\cite{mm}, Cor. 3.13), $\overline{R}=R/\Delta \cong R_{1}\times
R_{2}$, where $R_{1}$ is von Neumann regular right self-injective and $R_{2}$ is an
exchange ring with no non-zero nilpotent element. We have already shown in
Theorem \ref{main} that $R_{1}$ is twin-good. Since, $R_{2}$
has no non-zero nilpotent element, each idempotent in $R_{2}$ is central. Now we proceed to show that $R_2$ is also twin-good. Assume to the contrary that there exists an element $a\in R_{2}$ which is not twin-good. Then as in the
proof of Theorem \ref{main}, we find an ideal $I$ of $R_{2}$ such that $x=a+I\in
R_{2}/I$ is not twin-good in $R_{2}/I$ and $R_{2}/I$ has no central
idempotent. This implies that $R_{2}/I$ is an exchange ring without any
non-trivial idempotent, and hence it must be local. If $S=R_{2}/I$ then $x+J(S)$ is not twin-good in $S/J(S)$, which is a division ring.
Therefore, $S/J(S)\cong \mathbb{Z}_{2}$, or $\mathbb Z_3$, a contradiction. Hence, every
element of $R_{2}$ is twin-good. Therefore, every element of $\overline{R}$ is twin-good and hence $R$
is twin-good. This completes the proof.
\end{proof}

\begin{corollary}
The endomorphism ring $R=\End(M_S)$ of a continuous module $M_S$ is twin-good if $R$ has no factor isomorphic to $\mathbb{Z}_{2}$ or $\mathbb Z_3$.
\end{corollary}

\begin{proof}
It follows from the above corollary in view of the fact that a continuous module is quasi-continuous and also has exchange property.
\end{proof}

\noindent A module $M$ is called \textit{cotorsion} if every short exact sequence $0\longrightarrow
M\longrightarrow E\longrightarrow F\longrightarrow 0$ with $F$ flat, splits. It is known due to Guil Asensio and Herzog that if $M$ is a
flat cotorsion right $R$-module and $S=\End(M_{R}),$ then $S/J(S)$ is a von Neumann regular right self-injective ring (see \cite{gh}). As a consequence, we have the following

\begin{corollary}
The endomorphism ring $R=\End(M_S)$ of a flat cotorsion (in particular, pure injective) module $M_S$ is twin-good if $R$ has no factor ring isomorphic to $\mathbb{Z}_{2}$ or $\mathbb Z_3$.
\end{corollary}

\noindent A right $R$-module $M$ is called a {\it Harada module} if it has a decomposition $M=\oplus_{i\in \mathcal I}M_i$ with $\End(M_i)$ a local ring for each $i$, such that the decomposition complements direct summands of $M$ (that is, if $A$ is direct  summand of $M$ then there exists $\mathcal J\subseteq \mathcal I$ such that $M = A\oplus (\oplus_{j\in \mathcal J}M_j)$. 

\bigskip

\noindent Consider the following conditions on a module $N$;

\bigskip

\noindent ($D1$): For every submodule $A$ of $N$, there exists a decomposition $N=N_1\oplus N_2$ such that $N_1\subseteq A$ and $N_2\cap A$ is small in $N$.

\bigskip

\noindent ($D2$): If $A$ is a submodule of $N$ such that $N/A$ is isomorphic to a direct summand of $N$, then $A$ is a direct summand of $N$.

\bigskip

A right $R$-module $N$ is called a {\it discrete module} if $N$ satisfies the conditions $D1$ and $D2$. It is well known that every discrete module is a Harada module. 

\begin{corollary}
The endomorphism ring $R=\End(M_S)$ of a Harada module $M_S$ is twin-good if $R$ has no factor ring isomorphic to $\mathbb{Z}_{2}$ or $\mathbb Z_3$.
\end{corollary}

\begin{proof}
It is known due to Kasch \cite{Kasch} that $R/J(R)$ is a direct product of right full linear rings and hence the corollary follows from Theorem \ref{main}.
\end{proof}

\bigskip

\bigskip

\end{document}